\documentclass[12pt, reqno]{amsart} 

\title{Continuous projections onto ideal \\ convergent sequences}

\author[P.~Leonetti]{Paolo Leonetti}
\address{Department of Statistics, Universit\`a ``L. Bocconi'', via Roentgen 1, 20136 Milan, Italy}
\email{leonetti.paolo@gmail.com}
\urladdr{\url{https://sites.google.com/site/leonettipaolo/}} 

\keywords{Meager ideal, $\mathcal{I}$-maximal almost disjoint family, complementability, asymptotic density zero sets, $\mathcal{I}$-convergent sequence.}

\subjclass[2010]{Primary: 40A35, 46B03. Secondary: 54A20, 46B26.}

\usepackage[T1]{fontenc}
\usepackage{amsmath}
\usepackage{amssymb}
\usepackage{amsthm}
\usepackage[left=3.5cm, right=3.5cm, paperheight=11.8in]{geometry}
\usepackage{hyperref}
\usepackage{fancyhdr}
\usepackage{bm}
\usepackage{enumitem}
\usepackage{comment}
\usepackage{nicefrac}
\usepackage{mathrsfs}
\usepackage{graphicx}
\usepackage[utf8]{inputenc}
\usepackage{cancel}
\usepackage{mathtools}

\AtBeginDocument{%
   \def\MR#1{}
}

\newtheorem{thm}{Theorem}[section]
\newtheorem{cor}[thm]{Corollary}
\newtheorem{lem}[thm]{Lemma}

\theoremstyle{definition} 
\newtheorem{defi}[thm]{Definition}
\let\olddefi\defi
\renewcommand{\defi}{\olddefi\normalfont}
\newtheorem{question}{Question}
\let\oldquestion\question
\renewcommand{\question}{\oldquestion\normalfont}
\newtheorem{example}[thm]{Example}
\let\oldexample\example
\renewcommand{\example}{\oldexample\normalfont}
\newtheorem{rmk}[thm]{Remark}
\let\oldrmk\rmk
\renewcommand{\rmk}{\oldrmk\normalfont}

\hypersetup{
    pdftitle={Projections onto ideal convergent sequences},
    pdfauthor={Paolo Leonetti},
    pdfmenubar=false,
    pdffitwindow=true,
    pdfstartview=FitH,
    colorlinks=true,
    linkcolor=blue,
    citecolor=green,
    urlcolor=cyan
}

\providecommand{\MR}[1]{}

\providecommand{\MR}{\relax\ifhmode\unskip\space\fi MR }

\providecommand{\href}[2]{#2}

\begin{document}

\maketitle
\thispagestyle{empty}

\begin{abstract} Let $\mathcal{I}\subseteq\mathcal{P}(\omega)$ be a meager ideal. 
Then there are no continuous projections from $\ell_\infty$ onto the set of bounded sequences which are $\mathcal{I}$-convergent to $0$. 
In particular, it follows that the set of bounded sequences statistically convergent to $0$ is not isomorphic to $\ell_\infty$. 
\end{abstract}


\section{Introduction}
A closed subspace $X$ of a Banach space $B$ is said to be complemented in $B$ if 
there exists a continuous projection from $B$ onto $X$.  
It is known that $c_0$, the space of real sequences convergent to $0$, is not complemented in $\ell_\infty$, 
cf. \cite{MR0005777, MR1533692}. 
The aim of this note is to show the ideal analogue of this result. 

Let $\mathcal{I}\subseteq \mathcal{P}(\omega)$ be an ideal, that is, a family closed under subsets and finite unions. 
It is also assumed that 
$\mathrm{Fin}:=[\omega]^{<\omega} \subseteq \mathcal{I}$ and $\omega \notin \mathcal{I}$. 
Set $\mathcal{I}^+:=\mathcal{P}(\omega)\setminus \mathcal{I}$. 
In particular, each $\mathcal{I}$ can be regarded as a subset of the Cantor space $2^\omega$ with the product topology, so we can speak of Borel ideals, $F_\sigma$ ideals, etc. 
An ideal $\mathcal{I}$ is said to be a P-ideal if it is $\sigma$-directed modulo finite sets, i.e., for each sequence $(A_n)$ in $\mathcal{I}$ there exists $A \in \mathcal{I}$ such that $A_n \setminus A$ is finite for all $n \in \omega$. 
We refer to \cite{MR2777744} for a recent survey on ideals and filters.

A real sequence $(x_n)$ is said to be $\mathcal{I}$-convergent to $y$ if $\{n: x_n \notin U\} \in \mathcal{I}$ for all neighborhoods $U$ of $y$. We denote by $c(\mathcal{I})$ [resp. $c_{0}(\mathcal{I})$] the space of real sequences which are $\mathcal{I}$-convergent [resp. $\mathcal{I}$-convergent to $0$]. The set of bounded real $\mathcal{I}$-convergent sequences has been studied, e.g., in \cite{MR2735533, Filipow18, MR2181783}. 
By an easy modification of \cite[Theorem 2.3]{MR2181783}, 
$c_0(\mathcal{I}) \cap \ell_\infty$ is a closed linear subspace of $\ell_\infty$ \textup{(}with the sup norm\textup{)}. 

The question addressed here, posed at the open problem session of the 45th Winter School in Abstract Analysis (Czech Republic, 2017), follows:
\begin{question}\label{q:original}
Is $c_{0}(\mathcal{I}) \cap \ell_\infty$ complemented in $\ell_\infty$? 
\end{question}

Before proving our main result, we recall the following:
\begin{lem}\label{lem:equivalence}
An infinite dimensional subspace $X$ of $\ell_\infty$ is complemented  in $\ell_\infty$ if and only if it is isomorphic to $\ell_\infty$.
\end{lem}
\begin{proof}
See \cite[Proposition 2.5.2 and Theorem 5.6.5]{MR2192298}. 
\end{proof}

Hence, Question \ref{q:original} can be reformulated as:
\begin{question}\label{q:new}
Is ${c}_{0}(\mathcal{I})\cap \ell_\infty$ isomorphic to $\ell_\infty$? 
\end{question}

We will prove that the answer is negative for a large class of ideals. To state our result, we recall that a family $\mathscr{A} \subseteq \mathcal{I}^+$ is said to be $\mathcal{I}$\emph{-maximal-almost-disjoint} (in short, $\mathcal{I}$-$\mathrm{mad}$) if $\mathscr{A}$ is a maximal family (with respect to inclusion) such that $A \cap B \in \mathcal{I}$ for all distinct $A,B \in \mathscr{A}$, so that for each $X \in \mathcal{I}^+$ there exists $A \in \mathscr{A}$ such that $X \cap A \in \mathcal{I}^+$. (The minimal cardinality $\mathfrak{a}(\mathcal{I})$ of an $\mathcal{I}$-$\mathrm{mad}$ has been studied in the literature: e.g., it is known that, 
if $\mathcal{I}$ is an analytic P-ideal, $\mathfrak{a}(\mathcal{I})>\omega$ if and only if $\mathcal{I}$ is $F_\sigma$, 
cf. \cite{MR823775, MR2537837}.)

Our main result follows:
\begin{thm}\label{thm:complementability}
Let $\mathcal{I}$ be an ideal for which there exists an uncountable $\mathcal{I}$-$\mathrm{mad}$ family. Then $c_{0}(\mathcal{I}) \cap \ell_\infty$ is not complemented in $\ell_\infty$.
\end{thm}

It can be shown that, if $\mathcal{I}$ is a meager ideal, 
there is an $\mathcal{I}$-$\mathrm{mad}$ family of cardinality $\mathfrak{c}$, see Lemma \ref{lem:meagerc} below. In particular 
\begin{cor}\label{cor:meager}
$c_{0}(\mathcal{I}) \cap \ell_\infty$ is not complemented in \textup{(}and not isomorphic to\textup{)} $\ell_\infty$ whenever $\mathcal{I}$ is meager.
\end{cor}

As an important example, the family of asymptotic density zero sets $\mathcal{Z}:=\{S\subseteq \omega: |S\cap [1,n]|/n\to 0\}$ is an analytic P-ideal, hence meager. Therefore:
\begin{cor}\label{cor:stat}
The set of bounded real sequences statistically convergent to $0$ \textup{(}i.e., $c_0(\mathcal{Z})$\textup{)} is not is isomorphic to $\ell_\infty$.
\end{cor}

Lastly, we obtain an analogue of the main result in \cite{MR1532370} (for summability matrices):
\begin{cor}\label{cor:corsummability}
$c$ is complemented in $c(\mathcal{I}) \cap \ell_\infty$ if and only if $\mathcal{I}=\mathrm{Fin}$.
\end{cor}

It is worth noting that Theorem \ref{thm:complementability} cannot be extended to all ideals $\mathcal{I}$. Indeed, if $\mathcal{I}$ is maximal, then the set of bounded $\mathcal{I}$-convergent sequences, which is isomorphic to $c_0(\mathcal{I}) \cap \ell_\infty$, is exactly $\ell_\infty$.

\section{Preliminaries and Proofs}

Thanks to Lemma \ref{lem:equivalence}, a negative question to Question \ref{q:original} would follow if $c_0(\mathcal{I}) \cap \ell_\infty$ was separable (indeed $\ell_\infty$ is nonseparable, hence they cannot be isomorphic). However, this works only if $\mathcal{I}=\mathrm{Fin}$:

\begin{lem}\label{lem:c0starsepar}
$c_0(\mathcal{I})$ is separable if and only if $\mathcal{I}=\mathrm{Fin}$.
\end{lem}
\begin{proof}
The if part is known. 
Conversely, let us suppose that there exists $A \in \mathcal{I} \cap [\omega]^\omega$. For each $X\subseteq \omega$ and $\varepsilon>0$, let $B(\bm{1}_X,\varepsilon)$ be the open ball with center $\bm{1}_X$ and radious $\varepsilon$. The collection 
$
\mathscr{B}:=\left\{B(\bm{1}_X,1/2): X \in [A]^\omega\right\}
$ 
is an uncountable family of nonempty open sets which are pairwise disjoint, hence $c_0(\mathcal{I})$ is not separable.
\end{proof}

At this point, recall the following characterization, see \cite{MR579439} and \cite[Theorem 4.1.2]{MR1350295}:
\begin{lem}\label{lem:meagercharac}
$\mathcal{I}$ is a meager ideal if and only if there exists a finite-to-one function $f:\omega \to \omega$ such that $f^{-1}(A) \in \mathcal{I}$ if and only if $A$ is finite.
\end{lem}
In other words, the second condition is $\mathrm{Fin}\le_{\mathrm{RB}}\mathcal{I}$, where $\le_{\mathrm{RB}}$ is the Rudin--Blass ordering. This is sufficient to prove the existence of an uncountable $\mathcal{I}$-$\mathrm{mad}$ family:
\begin{lem}\label{lem:meagerc}
There exists an $\mathcal{I}$-$\mathrm{mad}$ family of cardinality $\mathfrak{c}$, provided $\mathcal{I}$ is meager.
\end{lem}
\begin{proof}
It is known that there is a $\mathrm{Fin}$-$\mathrm{mad}$ family $\mathscr{A}$ of cardinality $\mathfrak{c}$, cf. \cite{MR1533692}. 
Then, thanks to Lemma \ref{lem:meagercharac}, there exists a finite-to-one function $f:\omega \to \omega$ such that $f^{-1}(A) \in \mathcal{I}$ if and only if $A$ is finite, hence $\{f^{-1}(A):A \in \mathscr{A}\}$ is the claimed $\mathcal{I}$-$\mathrm{mad}$ family.
\end{proof}

Let us prove our main result:
\begin{proof}[Proof of Theorem \ref{thm:complementability}]
Let us suppose for the sake of contradiction that $c_0(\mathcal{I}) \cap \ell_\infty$ is complemented in $\ell_\infty$ and denote by $$\pi:\ell_\infty \to c_0(\mathcal{I}) \cap \ell_\infty$$ the canonical projection. Define $T:=I-\pi$, hence $T$ is bounded linear operator such that $T(x)=0$ for each $x \in c_0(\mathcal{I}) \cap \ell_\infty$. Note also that, if $B \notin \mathcal{I}$, then $\bm{1}_B$ is a bounded sequence which is not $\mathcal{I}$-convergent to $0$, hence $\pi(\bm{1}_B) \neq \bm{1}_B$ and $T(\bm{1}_B) \neq 0$.

At this point, let $(A_j: j \in J)$ be an uncountable $\mathcal{I}$-$\mathrm{mad}$ family, which exists by hypothesis. 
We are going to show that there exists $j \in J$ such that $T(\bm{1}_{A_j})=0$, which is impossible since $A_j \in \mathcal{I}^+$. Indeed, let us suppose that, for each $j \in J$, there exists $x_j=(x_{j,n}) \in \ell_\infty$ supported on $A_j$ with $T(x_j)\neq 0$ and, without loss of generality, $\|x_j\|_\infty=1$. It follows that there exists $m,k \in \omega$ such that $\tilde{J}:=\{j \in J: |x_{j,m}|\ge 2^{-k}\}$ is uncountable. Also, by possibly replacing $x_j$ with $-x_j$, let us suppose without loss of generality that $x_{j,m}>0$ for all $j \in \tilde{J}$.

For each nonempty finite set $F \subseteq \tilde{J}$, define $s_F=(s_{F,n}):=\sum_{j \in F}x_j$. In particular, 
\begin{equation}\label{eq:final}
\|T(s_F)\|_\infty \ge s_{F,m} \ge |F|2^{-k}.
\end{equation}
Note also that $I:=\bigcup (A_i \cap A_j)$, where the sum is extended over all distinct $i,j \in F$, belongs to $\mathcal{I}$. This implies that the sequence $s_F \upharpoonright I$ is $\mathcal{I}$-convergent to $0$, hence $T(s_F)=T(s_F\upharpoonright I^c)$. Therefore
$$
\|T(s_F)\|_\infty=\|T(s_F\upharpoonright I^c)\|_\infty\le \|T\| \cdot \|s_F\upharpoonright I^c\|_\infty \le \|T\|,
$$
which, together with \eqref{eq:final}, implies $|F| \le 2^k\|T\|$. This contradicts the fact the $\tilde{J}$ is infinite.
\end{proof}

\begin{proof}[Proof of Corollary \ref{cor:corsummability}]
There is nothing to prove if $\mathcal{I}=\mathrm{Fin}$. Conversely, fix $I \in \mathcal{I}\setminus \mathrm{Fin}$ and define $X:=\{x \in \ell_\infty: x_i\neq 0\text{ only if }i \in I\}$ and $Y:=X \cap c_0$. It is clear that $$c\subseteq Y\subseteq X\subseteq c(\mathcal{I}) \cap \ell_\infty$$ and that $X$ and $Y$ are isometric to $\ell_\infty$ and $c_0$, respectively. Hence, it is known that $c$ can be projected continuously onto $Y$, let us say through $T$, see \cite{MR0005777}. To conclude the proof, let us suppose that there exists a continuous projection $H: c(\mathcal{I})\cap \ell_\infty \to c$. Then the restriction $T\circ H\upharpoonright X$ is a continuous projection $\ell_\infty \to c_0$. This contradicts Theorem \ref{thm:complementability} (in the case $\mathcal{I}=\mathrm{Fin}$).
\end{proof}

\subsection{Acknowledgments.} The author is grateful to Tommaso Russo (Universit\`a degli Studi di Milano, IT) for suggesting Question \ref{q:original} and Lemma \ref{lem:equivalence}.

\bibliographystyle{amsplain}
\bibliography{ideale}

\providecommand{\MR}[1]{}
\providecommand{\bysame}{\leavevmode\hbox to3em{\hrulefill}\thinspace}
\providecommand{\MR}{\relax\ifhmode\unskip\space\fi MR }
\providecommand{\MRhref}[2]{%
  \href{http://www.ams.org/mathscinet-getitem?mr=#1}{#2}
}
\providecommand{\href}[2]{#2}
\begin{thebibliography}{10}

\bibitem{MR2192298}
F.~Albiac and N.~J. Kalton, \emph{Topics in {B}anach space theory}, Graduate
  Texts in Mathematics, vol. 233, Springer, New York, 2006. \MR{2192298}

\bibitem{MR2735533}
A.~Bartoszewicz, S.~G\l{}ab, and A.~Wachowicz, \emph{Remarks on ideal
  boundedness, convergence and variation of sequences}, J. Math. Anal. Appl.
  \textbf{375} (2011), no.~2, 431--435. \MR{2735533}

\bibitem{MR1350295}
T.~Bartoszy\'{n}ski and H.~Judah, \emph{Set theory}, A K Peters, Ltd.,
  Wellesley, MA, 1995, On the structure of the real line. \MR{1350295}

\bibitem{MR823775}
J.~E. Baumgartner, \emph{Iterated forcing}, Surveys in set theory, London Math.
  Soc. Lecture Note Ser., vol.~87, Cambridge Univ. Press, Cambridge, 1983,
  pp.~1--59. \MR{823775}

\bibitem{MR2537837}
B.~Farkas and L.~Soukup, \emph{More on cardinal invariants of analytic
  {$P$}-ideals}, Comment. Math. Univ. Carolin. \textbf{50} (2009), no.~2,
  281--295. \MR{2537837}

\bibitem{Filipow18}
R.~Filip\'{o}w and J.~Tryba, \emph{Ideal convergence versus matrix
  summability}, Studia Math., to appear.

\bibitem{MR2777744}
M.~Hru\v{s}\'{a}k, \emph{Combinatorics of filters and ideals}, Set theory and
  its applications, Contemp. Math., vol. 533, Amer. Math. Soc., Providence, RI,
  2011, pp.~29--69. \MR{2777744}

\bibitem{MR2181783}
P.~Kostyrko, M.~Ma\v{c}aj, T.~\v{S}al\'{a}t, and M.~Sleziak,
  \emph{$\mathscr{I}$-convergence and extremal $\mathscr{I}$-limit points},
  Math. Slovaca \textbf{55} (2005), no.~4, 443--464. \MR{2181783}

\bibitem{MR1532370}
J.~Lindenstrauss, \emph{Mathematical {N}otes: {A} {R}emark {C}oncerning
  {P}rojections in {S}ummability {D}omains}, Amer. Math. Monthly \textbf{70}
  (1963), no.~9, 977--978. \MR{1532370}

\bibitem{MR0005777}
A.~Sobczyk, \emph{Projection of the space {$(m)$} on its subspace {$(c_0)$}},
  Bull. Amer. Math. Soc. \textbf{47} (1941), 938--947. \MR{0005777}

\bibitem{MR579439}
M.~Talagrand, \emph{Compacts de fonctions mesurables et filtres non
  mesurables}, Studia Math. \textbf{67} (1980), no.~1, 13--43. \MR{579439}

\bibitem{MR1533692}
R.~Whitley, \emph{Mathematical {N}otes: {P}rojecting {$m$} onto {$c_0$}}, Amer.
  Math. Monthly \textbf{73} (1966), no.~3, 285--286. \MR{1533692}

\end{thebibliography}

\end{document}